\documentclass[11pt,a4paper]{amsart}
\usepackage{amsmath, amssymb,amsthm}
\usepackage{color, hyperref}

\DeclareSymbolFont{bbold}{U}{bbold}{m}{n}
\DeclareSymbolFontAlphabet{\mathbbold}{bbold}

\newtheorem{thm}{Theorem}[section]
\newtheorem{prop}[thm]{Proposition}
\newtheorem{lemma}[thm]{Lemma}
\newtheorem{cor}[thm]{Corollary}

\newtheorem*{thm*}{Theorem}
\theoremstyle{definition}
\newtheorem{defn}[thm]{Definition}
\newtheorem{rem}[thm]{Remark}

\newcommand{\Q}{\mathbb Q}
\newcommand{\Z}{\mathbb Z}
\newcommand{\R}{\mathbb R}
\newcommand{\C}{\mathbb C}
\newcommand{\F}{\mathbb F}
\newcommand{\h}{\mathcal H}
\newcommand{\A}{\mathbb A}

\newcommand{\1}{\mathbbold 1}

\newcommand{\fq}{\mathfrak{q}}
\newcommand{\fo}{\mathfrak{o}}
\newcommand{\fC}{\mathfrak{C}}
\newcommand{\fO}{\mathfrak{O}}
\newcommand{\fP}{\mathfrak{P}}
\newcommand{\fp}{\mathfrak{p}}
\newcommand{\fD}{\mathfrak{D}}
\newcommand{\fN}{\mathfrak{N}}

\DeclareMathOperator{\Alb}{Alb}
\DeclareMathOperator{\Aut}{Aut}
\DeclareMathOperator{\GL}{GL}
\DeclareMathOperator{\Lie}{Lie}
\DeclareMathOperator{\Ind}{Ind}

\DeclareMathOperator{\rH}{H}

\title{Arithmetic quotients of the complex ball and a conjecture of Lang}

\author{Mladen Dimitrov \and Dinakar Ramakrishnan}

\address{Universit{\'e}  Lille 1, UMR  8524, UFR Math{\'e}matiques,
59655 Villeneuve d'Ascq Cedex, France}
\email{mladen.dimitrov@gmail.com}

\address{Mathematics 253-37, Caltech,  Pasadena, CA 91125, USA}
\email{dinakar@caltech.edu}

\date{}

\begin{document}

\maketitle

\section*{Introduction}

Let $F$ be a totally real number field  of degree $d$ and ring of integers $\fo$, and let
 $M$ be a  totally imaginary quadratic extension of $F$ with  ring of integers $\fO$.
  Let $G$ be a unitary group over $F$ defined by a hermitian form on $M^{n+1}$  of signature $(n,1)$ at one infinite place $\iota$ and $(n+1,0)$ or $(0,n+1)$ at the others. A  subgroup $\Gamma\subset G(F)$ is arithmetic if it is commensurable with $G(\fo)$ -- the stabilizer in $G(F)$ of  $\fO^{n+1}$ -- and we  will  denote by $Y_\Gamma$ the quotient of the $n$-dimensional complex hyperbolic space by 
    the natural action of  $\iota(\Gamma)\subset G(F_\iota)={\rm U}(n,1)$. If $F\neq \Q$, then the hermitian form is anisotropic and  $Y_\Gamma$ is a projective variety defined  over a number field  (see  Proposition \ref{prop-yau}).

A  projective variety $X$ over $\C$  is said to be Mordellic if it has only a finite number of rational points in every finitely generated field extension of
$\Q$ over which $X$ is defined. Lang conjectured in \cite[Conjecture VIII.1.2]{lang}  that $X$ is  Mordellic if and only if the corresponding analytic space $X(\C)$ is
hyperbolic, meaning that any holomorphic map $\C\to X(\C)$ is constant, which by   Brody \cite{brody} is equivalent to requiring  the Kobayashi semi-distance on $X(\C)$ to be a metric.
It is a consequence of a conjecture of Ullmo (see \cite[Conjecture 2.1]{ullmo}) that a projective variety $X$ defined over a number field $k$ is Mordellic if it
is arithmetically Mordellic, meaning that it has only a finite number of rational points in every finite extension of $k$.

Our first result establishes that many arithmetic compact surfaces  previously only known to be arithmetically Mordellic by   \cite[Th\'eor\`eme 3.2]{ullmo} are in fact Mordellic. To state it precisely we need to 
fix a Hecke character $\lambda$ of $M$ as in  Definition \ref{CM-type}. The existence of such characters is known (see  Lemma \ref{CM-Hecke}).
Denote by $\fC$  the conductor of $\lambda$ and, if the extension $M/F$ is everywhere unramified, we multiply   $\fC$ by any prime $\fq$ of $F$ which does not split in $M$. Moreover, fix an auxiliary prime $\fp$ of $F$ which splits in $M$ and is relatively prime to  $\fC$.
Finally, for every  ideal $\fN\subset \fO$ we consider the standard congruence subgroups $\Gamma_0(\fN)$, $\Gamma_1(\fN)$ and $\Gamma(\fN)$ of $G(F)$ (see Definition  \ref{defn-congruence}).

\begin{thm}\label{picard-compact}
Let $n=2$ and  $G$ over $F$ as above. Then for every choice of $(\fC,\fp)$, and for  any
torsion free subgroup $\Gamma\subset \Gamma_1(\fC)\cap \Gamma_0(\fp)$ of finite index,
$Y_{\Gamma}$ is   Mordellic.
\end{thm}

A consequence of this is that   for any arithmetic subgroup $\Gamma\subset G(F)$ there exists a finite explicit cover of $Y_\Gamma$ which is Mordellic.
Note  also that even though the theorem  only concerns arithmetic subgroups, because $F$ and $M$ can vary, it can be applied to infinitely many pairwise non-commensurable cocompact discrete subgroups in ${\rm U}(2,1)$. In order to apply our method  to the analogous case of a unitary group $G''$ defined by a division algebra of dimension $9$ over $M$ with an involution of the second kind, one would need to find  a (cocompact) arithmetic subgroup  $\Gamma\subset G''(F)$ such that the Albanese of $Y_\Gamma$ is non-zero. This is an open question for any $G''$ since, in contrast to our case, it is known by Rogawski \cite{rogawski1}
that the Albanese of $Y_\Gamma$ is zero for any congruence subgroup $\Gamma\subset G''(F)$.

While  Ullmo's approach uses the Shafarevich conjecture, ours
   is based instead on the Mordell-Lang conjecture proved by Faltings  \cite{faltings}  and on the key Proposition \ref{positivity},   which we hope is  of independent interest.

Consider now the case when the hermitian form is isotropic,  which necessarily implies that $F=\Q$ and $M$ is imaginary quadratic.  Then $Y_\Gamma$ is not compact and, for $\Gamma$ arithmetic, we denote by  $Y_\Gamma^\ast$ the  Baily-Borel compactification which is a normal, projective variety of dimension $n$.
A smooth toroidal compactification $X_\Gamma$ of  $Y_\Gamma$ can be defined over a number field (see \cite{faltings-rigid}), and it is not hyperbolic even if $\Gamma$ is torsion-free;
for example, if $n=2$, then   $X_\Gamma$ is  a union of $Y_\Gamma$ with a finite number of elliptic curves -- one above each cusp of  $Y_\Gamma^\ast$.
However, by a result of Tai and Mumford \cite[\S4]{mumford}, $X_\Gamma$ is  of  general type for $\Gamma$ sufficiently small. The  {\it Bombieri-Lang conjecture} asserts then that the points of $X_\Gamma$ over any finitely generated field extension of $\Q$ over which
  $X_\Gamma$  is defined  are not Zariski dense. 
  We prove this in Proposition \ref{bombieri-lang} which  allows us to solve an alternative of Ullmo and Yafaev \cite{ullmo-yafaev} regarding the Lang locus of $Y_\Gamma^\ast$.

\begin{thm} \label{picard-alternative}
For all $\Gamma \subset G(\Q)$  arithmetic and  sufficiently small,   $Y_\Gamma^\ast$ is  arithmetically Mordellic.
\end{thm}

Keeping the assumption that  $M$ is  imaginary quadratic, say  of fundamental discriminant $-D$, let us now suppose in addition that  $n=2$.  The corresponding locally symmetric spaces $Y_\Gamma$ are  called Picard modular surfaces. We state here our main theorem.

\begin{thm}  \label{picard-open} Let $\fD=
\begin{cases}
3 \fO& \text{, if } D=3,  \\
\sqrt{-D} \fO& \text{, if }  D\neq 3 \text{ is odd, } \\
2\sqrt{-D}\fO & \text{, if }  8\text{ divides }  D,  \\
\sqrt{-D}\fP_2 & \text{, otherwise, where  $\fP_2$ is the prime of $M$ above $2$. }
\end{cases}$
 \begin{enumerate}
\item Let $ \Gamma= \begin{cases} \Gamma_1(\fD)& \text{, if } 
D\notin \{3,4,7,8,11,15,19,20,23,24,31,39,43,47, 67, 71, 163\}, \\
\Gamma(\fD)& \text{, if }  D\in  \{8,15,20,23,24,31,39,47, 71\} , \\
 \Gamma(\fD^2)&  \text{, if }  D\in  \{3,4,7,11,19, 43, 67, 163\}. 
\end{cases} $
Then $Y_{\Gamma}^\ast$ is  Mordellic, while $X_{\Gamma}$ is a minimal surface of general type.

\item Let $N>2$ be a prime inert in $M$ and not equal to $3$ when  $D=4$.
Then  $Y_{\Gamma(N)\cap \Gamma_1(\fD)}^\ast$ is  Mordellic, while
$X_{\Gamma(N)\cap \Gamma_1(\fD)}$ is a minimal surface of general type.
\end{enumerate}
\end{thm}
The  cases $D\in  \{15,20,23,24,31,39,47,71\}$ in (i) of  Theorem \ref{picard-open}
depend on a preprint of D\v{z}ambi\'c \cite{dzambic} which is
being considered for publication elsewhere (see the proof for details).

At the heart of our proof stand some arithmetical computations  using  certain key theorems of Rogawski \cite{rogawski1,rogawski2}.
They yield, for each imaginary quadratic field $M$, an explicit congruence subgroup  $\Gamma$ such that the smooth compactification $X_\Gamma$ does not admit a dominant map to its Albanese variety.
A  geometric ingredient of the proof is a result of Holzapfel {\it et al}  that  $X_{\Gamma}$ is of general type, though not hyperbolic, implying by a theorem of Nadel \cite{nadel}  that any curve of genus $\leq 1$ on it is contained in the compactifying  divisor.

If there is anything new in our approach, it lies in the systematic use of the modern theory of automorphic representations in Diophantine geometry.

\small
\subsection*{Acknowledgments}

We would like to thank Don Blasius, Jean-Fran\c{c}ois Dat,
Najmuddin Fakhruddin, Dick Gross, Barry Mazur,  Matthew Stover and  Shing-Tung Yau for helpful conversations. In fact it was Fakhruddin who suggested our use of the  Mordell-Lang 
conjecture for abelian varieties. Needless to say, this Note owes much to the deep results of Faltings. In addition, we thank the referee and 
Blasius for   their corrections and suggestions which led to an improvement of  the presentation. 
Thanks are also due to Serge Lang (posthumously), and to John Tate, for getting one of us interested in the conjectural Mordellic property of hyperbolic varieties. Finally, we are also happy to acknowledge partial support  the following sources:
the Agence Nationale de la Recherche grants ANR-10-BLAN-0114 and ANR-11-LABX-0007-01 for the first author (M.D.), and from the NSF grant DMS-1001916 for the second author (D.R.).
\normalsize

\section{Basics: lattices, general type and neatness} \label{general}

For any integer $n>1$,  let $\h^n_\C$ be the $n$-dimensional complex hyperbolic space, represented by the unit ball in $\C^n$ equipped with the Bergman metric of constant holomorphic sectional curvature $-4/(n+1)$, on which the real Lie group ${\rm U}(n,1)$ acts in a natural way.

 Given a  lattice $\Gamma\subset{\rm U}(n,1)$ we denote by
 $\bar\Gamma=\Gamma/\Gamma\cap {\rm U}(1)$ its image in the adjoint group ${\rm PU}(n,1)={\rm U}(n,1)/{\rm U}(1)$, where ${\rm U}(1)$ is centrally embedded in ${\rm U}(n,1)$. 
 Conversely any  lattice $\bar\Gamma \subset{\rm PU}(n,1)={\rm PSU}(n,1)$ is the image of a
 lattice in ${\rm U}(n,1)$, namely  the lattice ${\rm U}(1) \bar\Gamma \cap {\rm SU}(n,1)$. 
 We consider the quotient $Y_\Gamma=Y_{\bar\Gamma}=  \bar\Gamma\backslash \h^n_\C$.

\begin{lemma}\label{covers} Let  $\Gamma$ be a lattice in  ${\rm U}(n,1)$.
\begin{enumerate}
\item The analytic variety $Y_\Gamma$ is an orbifold and  one has the following implications:
$$ \Gamma \text{  neat} \Rightarrow \Gamma\text{ torsion-free} \Rightarrow
\bar\Gamma\text{  torsion-free} \Rightarrow Y_\Gamma \text{ is a  hyperbolic manifold.}$$

\item Assume that $\bar\Gamma$ is torsion-free. Then the natural projection $\h^n_\C \to Y_\Gamma$ is an etale covering with deck transformation group $\bar\Gamma$. \end{enumerate}
\end{lemma}

\begin{proof}
The stabilizer in ${\rm U}(n,1)$ of any point of $\h^n_\C$ is a compact group, hence its intersection with the discrete subgroup $\Gamma$ is finite, showing that $Y_\Gamma$ is an orbifold.

Recall that $\Gamma$ is neat if the subgroup of $\C^\times $ generated by the eigenvalues of any $\gamma\in \Gamma$ is torsion-free.
 In particular $\Gamma$ is torsion-free. Since $\Gamma\cap {\rm U}(1)$ is finite, this implies that $\bar\Gamma$ is torsion-free too.
Under the latter assumption, $\Gamma\cap {\rm U}(1)$ acts trivially on $\h^n_\C$, and  $\bar\Gamma$ acts freely and properly discontinuously on it, hence $Y_\Gamma$ is a manifold. Since $\h^n_\C$  is simply connected, it is  a universal covering space of $Y_\Gamma$ with group $\bar\Gamma$. Hence any holomorphic map from $\C$ to $Y_\Gamma$ lifts to a holomorphic map from $\C$ to $\h^n_\C$ which must be constant because $\h^n_\C$ has negative curvature. Thus $Y_\Gamma$ is hyperbolic.
\end{proof}

Deligne's classification \cite{deligne} of Shimura varieties  implies, when  $\Gamma$ is a congruence subgroup, that  $Y_\Gamma$ admits an  embedding in a Shimura variety.  Hence, by
Shimura's theory of canonical models, $Y_\Gamma$ can be defined over a finite abelian extension of the reflex field $M$.

We claim that this is also true for $\Gamma$ arithmetic, when sufficiently small. Indeed any such $Y_\Gamma$ is a finite unramified  cover of a congruence quotient $Y_{\Gamma'}$ which we have seen is  defined over a number field. 
By Grothendieck, the finite index subgroup $\bar\Gamma$ of the topological fundamental group $\bar\Gamma'$ of 
$Y_{\Gamma'}(\C)$ gives rise to a finite index subgroup of the algebraic fundamental group of  $Y_{\Gamma'}$, yielding
a finite algebraic (etale) map from a model of $Y_\Gamma$ to $Y_{\Gamma'}$. 

In the cocompact case, this remains true even when $\Gamma$ is not arithmetic.

\begin{prop}  \label{prop-yau}
Assume that $\bar\Gamma$ is cocompact and torsion-free. Then the projective variety  $Y_{\bar\Gamma}$ is of general type and  can be defined over a number field.
\end{prop}

\begin{proof}
The existence of the positive Bergman metric on $\h^n_\C$  implies  by the Kodaira embedding theorem that any quotient by a free action such as $Y_{\bar\Gamma}$ has ample canonical bundle, which results in $Y_{\bar\Gamma}$ being of general type; it even implies that any subvariety of  $Y_{\bar\Gamma}$ is of general type. For surfaces one may alternately use the hyperbolicity of  $Y_{\bar\Gamma}$ to rule out all the cases in the Enriques-Kodaira classification  where the  Kodaira dimension is less than $2$, thus showing that $Y_{\bar\Gamma}$ is of general type.

Calabi and Vesentini \cite{calabi-vesentini} have proved  that $Y_{\bar\Gamma}$ is locally rigid, hence by  Shimura \cite{shimura-rigid} it can be defined over a number field.

In order to highlight the importance of rigidity of compact ball quotients, we   provide  a short second proof when $n=2$ and $\Gamma$ is arithmetic, based on Yau's algebro-geometric characterization of compact K\"ahler surfaces covered by $\h^2_\C$.
Since $Y_{\bar\Gamma}$ has an ample canonical bundle it can be embedded in some projective space, hence is algebraic over $\C$ by Chow. Since $Y_{\bar\Gamma}$  is uniformized by $\h^2_\C$, the
Chern numbers $c_1, c_2$  of its complex tangent bundle
satisfy the relation $c_1^2=3c_2$. Since everything can be defined
algebraically,  for any automorphism $\sigma$ of $\C$, the variety $Y_{\bar\Gamma}^\sigma$ also has ample canonical bundle and  $c_1^{\sigma 2}=3c_2^\sigma$.   By a famous result of Yau  \cite[Theorem 4]{yau}, this is equivalent to the fact that
$Y_{\bar\Gamma}^\sigma$ may be realized as $\bar\Gamma^\sigma \backslash \h^2_\C$ for some cocompact  torsion-free lattice  $\bar\Gamma^\sigma$.

Since $\bar\Gamma$ is arithmetic, it has infinite index in its commensurator  in ${\rm PU}(2,1)$, denoted
by ${\rm Comm}(\bar\Gamma)$.  For every element $g\in {\rm Comm}(\bar\Gamma)$ there is a  Hecke correspondence
\begin{equation}\label{corr}
Y_{\bar\Gamma} \leftarrow Y_{\bar\Gamma\cap g^{-1}\bar\Gamma g}\overset{\sim}{\underset{g\cdot}{\longrightarrow}} Y_{g\bar\Gamma g^{-1} \cap \bar\Gamma }\to Y_{\bar\Gamma}
\end{equation}
and the correspondences for $g$ and $g'$ differ by an isomorphism
$Y_{g\bar\Gamma g^{-1} \cap \bar\Gamma }\overset{\sim}{\longrightarrow}
Y_{g'\bar\Gamma g'^{-1} \cap \bar\Gamma }$ over $Y_{\bar\Gamma}$ if and only if $g'\in \bar\Gamma g$.
By Chow \eqref{corr} is defined algebraically, hence yields a correspondence on $Y_{\bar\Gamma}^\sigma=Y_{\bar\Gamma^\sigma}$:
$$Y_{\bar\Gamma^\sigma} \leftarrow Y_{\bar\Gamma_1}\overset{\sim}{\longrightarrow} Y_{\bar\Gamma_2}\to Y_{\bar\Gamma^\sigma},$$
for some finite index subgroups  $\bar\Gamma_1$ and $\bar\Gamma_2$ of $\bar\Gamma^\sigma$.
By the universal property of the covering space $\h^2_\C$, the middle isomorphism is given by an element of $ g_\sigma \in  {\rm PU}(2,1)\simeq \Aut(\h^2_\C)$. Since $\Aut(\h^2_\C/Y_{\bar\Gamma_i})=\bar \Gamma_i$ ($i=1,2$),  it easily follows that $\bar\Gamma_2 = g_\sigma \bar\Gamma_1g_\sigma^{-1} $, and
by applying  $\sigma^{-1}$ one  sees that $\bar\Gamma_1=\bar\Gamma^\sigma\cap g_\sigma^{-1}\bar\Gamma^\sigma g_\sigma$. It  follows that $g_\sigma\in  {\rm Comm}_G(\Gamma^\sigma)$ and one can check that $g'_\sigma\in \bar\Gamma^\sigma g_\sigma$  if and only if  $g'\in \bar\Gamma g$.
Therefore
${\rm Comm}(\bar\Gamma^\sigma)/\bar\Gamma^\sigma\simeq {\rm Comm}(\bar\Gamma)/\bar\Gamma$
 is  infinite too, which by a major theorem of Margulis implies that $\bar\Gamma^\sigma$ is arithmetic,  providing an alternative proof of  a result of Kazhdan.

Thus $\Aut(\C)$ acts on the set of isomorphism classes of cocompact arithmetic quotients $Y_{\bar\Gamma}$, or equivalently, on the set of equivalence classes of cocompact arithmetic subgroups $\bar\Gamma$ (up to conjugation by an element of ${\rm PU}(2,1)$). The latter set is countable for the following reason. The group ${\rm U}(2,1)$  has  only countably many $\Q$-forms, classified by central simple algebras of dimension $9$ over  $M$, endowed with an involution of a second kind and verifying some conditions at infinity (see \cite[pp.~87-88]{platonov-rapinchuk}). Moreover,  there are only
countably many arithmetic subgroups for a given $\Q$-form, since those are all finitely generated and contained in their common commensurator, which is countable. 

Finally, by \cite[Corollary 2.13]{gonzalez}, the fact that $Y_{\bar\Gamma}$ has a countable orbit under the action of 
 $\Aut(\C)$ is equivalent to $Y_{\bar\Gamma}$ being defined over a number field.
\end{proof}

It is a well known fact that any  orbifold admits a finite cover which is a manifold. 
In view of Lemma \ref{covers}, the two lemmas below   provide such  covers  explicitly  for arithmetic quotients.

\begin{defn}  \label{defn-congruence}
For every  ideal $\fN\subset \fO$ we define  the  congruence subgroup $\Gamma(\fN)$ (resp.
$\Gamma_0(\fN)$, resp. $\Gamma_1(\fN)$) as the kernel (resp. the inverse image of upper triangular, resp. upper unipotent, matrices) of the composite homomorphism:
$$G(\fo) \hookrightarrow {\rm GL}(n+1,\fO) \to{\rm GL}(n+1,\fO/\fN).$$
\end{defn}

The following lemma is well-known (see \cite[Lemma 4.3]{holzapfel}).

\begin{lemma} \label{neat}
For any integer $N>2$ the group  $\Gamma(N)$ is neat.
\end{lemma}

\begin{lemma} \label{neatbis}
 Suppose that $n=2$ and that $M$ is an imaginary quadratic field of fundamental discriminant $-D\notin\{ -3,-4,-7,-8,-24\}$.
Then  $\Gamma_1(\sqrt{-D}\fO)$ is neat.
\end{lemma}

\begin{proof}
Suppose that the subgroup of $\C^\times$ generated by the eigenvalues of some $\gamma\in \Gamma_1(\sqrt{-D}\fO)$ contains a non-trivial root of unity. Note first that 
$\det(\gamma)\in \fO^\times \cap (1+\sqrt{-D}\fO)=\{1\}$. 

If $\gamma$ is elliptic then it is necessarily of finite order. Otherwise
$\gamma$ fixes a boundary point of $\h^2_\C\subset \mathbb{P}^2(\C)$ and is therefore conjugated in $\GL(3,\C)$ to a matrix of the form $\left(\begin{smallmatrix} \bar\alpha & * & * \\ 0 & \beta & * \\ 0 & 0 & \alpha^{-1} \end{smallmatrix}\right)$, where
$\beta$ is necessarily a root of unity. If $\beta=1$, then $\det(\gamma)=1$ implies that $\alpha\in \R$, 
 leading to  $\alpha=-1$. 
 Hence, in all cases, one may assume $\gamma$ has a non-trivial root of unity $\zeta$ as an eigenvalue.

By the Cayley-Hamilton theorem we have $[M(\zeta):M]\leq 3$ and since $D\neq 7$  we may assume (after possibly raising $\gamma$ to some power)  that $\zeta$ has order $2$ or $3$.
By the congruence condition,  each prime $p$ dividing $D$ has to divide also the norm of $\zeta-1$, hence
$D$ can be divisible only by the primes $2$ or $3$. Thus $D\in\{ 3,4,8,24\}$,  leading to a
contradiction.
\end{proof}

\section{Irregularity of arithmetic varieties}

Let  $z\mapsto \bar z$ be the non-trivial automorphism of $M/F$ and let
$\omega$ be the quadratic character   of $M/F$, viewed as a Hecke character of $F$.
Put  $M^1=\{z\in M^\times\mid z\bar z=1\}$, which we will view as an algebraic torus over $F$ and denote by  $\A_M^1$ its $\A_F$-points.

We denote by $q(X)$ the {\it irregularity} of $X$, given by the dimension of $\rH^0(X, \Omega^1_{X})$.

\subsection{Automorphic forms contributing to the irregularity}
Fix  a maximal compact subgroup $K_\infty \simeq ({\rm U}(n) \times {\rm U}(1)) \times {\rm U}(n+1)^{d-1}$ of the real linear   Lie group $G_\infty=G(F\otimes_{\Q}\R)\simeq {\rm U}(n,1)\times {\rm U}(n+1)^{d-1}$. 
Let  $\Gamma\subset G(F)$ be a  lattice such that $\bar\Gamma$ is torsion free. 

Since 
 $Y_\Gamma$ is the Eilenberg-MacLane space of $\overline\Gamma$, there is
 a decomposition:
\begin{equation}\label{decomposition}
\rH^1(Y_\Gamma, \C) \, \simeq \,  \rH^1(\bar\Gamma, \C)   \, \simeq \, \bigoplus_{\pi_\infty} \,
\rH^1(\Lie(G_\infty), K_\infty; \pi_\infty)^{\oplus m(\pi_\infty, \Gamma)},
\end{equation}
where $\pi_\infty$ runs over irreducible unitary representations of $G_\infty$ occurring in the discrete spectrum of $\mathrm{L}^2(\Gamma\backslash G_\infty)$ with multiplicity $m(\pi_\infty, \Gamma)$, and
$\rH^\ast(\Lie(G_\infty), K_\infty; \pi_\infty)$ is the relative Lie algebra cohomology.
 When $\Gamma$ is cocompact, the entire $\mathrm{L}^2$-spectrum is discrete and this decomposition follows from \cite[XIII]{borel-wallach}. When $\Gamma$ is  non-cocompact,
 one gets by \cite[\S4.4-4.5]{borel-casselman} such a decomposition, but  only for the $\mathrm{L}^2$-cohomology of $Y_\Gamma$. However, one knows (see \cite[\S1]{murty-ramakrishnan}) that
$\rH^1(Y_\Gamma, \C)$ is isomorphic to the middle intersection cohomology (in degree $1$) of
$Y_\Gamma^\ast$, which is in turn isomorphic to the $\mathrm{L}^2$-cohomology (in degree $1$) of $Y_\Gamma$.

By \cite[VI.4.11]{borel-wallach}  there are exactly two irreducible non-tempered unitary representations of ${\rm SU}(n,1)$ with trivial central character, denoted  $J_{1,0}$ and $J_{0,1}$, each of whose relative Lie algebra cohomology in degree $1$ does not vanish and  is in fact one dimensional.  Since ${\rm U}(n,1)$ is the product of its center with ${\rm SU}(n,1)$, $J_{1,0}$ and $J_{0,1}$ can be uniquely extended to representations  $\pi^+$ and $\pi^-$, say, of  ${\rm U}(n,1)$ with trivial central characters
(when $n=2$ those are the representations $J^\pm$ from \cite[p.178]{rogawski1}). It follows that at the distinguished Archimedean place $\iota$, where $G(F_\iota)={\rm U}(n,1)$, we have
$$\rH^1(\Lie({\rm U}(n,1)), {\rm U}(n)\times {\rm U}(1); \pi_\iota)=\begin{cases} \C \text{ , if } \pi_\iota=\pi^\pm, \\ 0\text{ , otherwise. }\end{cases}$$
Moreover the only irreducible unitary representation  with non-zero relative Lie algebra cohomology in degree $0$ is the trivial representation $\1$, which does not contribute in degree $1$;
in particular $\pi^\pm\neq \1$.
This  allows us to deduce from  \eqref{decomposition} the following formula
  \begin{equation}\label{h-formula}
\dim_\C \rH^1(Y_\Gamma, \C)= m(\pi^+\otimes \1^{\otimes d-1}, \Gamma)+
m(\pi^-\otimes \1^{\otimes d-1}, \Gamma),
\end{equation}
where $\pi^\pm$ are viewed as representations of $G(F_\iota)={\rm U}(n,1)$ and
 $\1^{\otimes d-1}$ denotes the trivial representation of ${\rm U}(n+1)^{d-1}$.

By  \cite[\S1]{murty-ramakrishnan},  $\rH^1(Y_\Gamma, \C)$ is isomorphic to $\rH^1(X_\Gamma, \C)$, hence admits a pure Hodge structure of weight $1$ and its dimension is given by $2q(X_\Gamma)$. 
In particular, the natural map $\rH^0(X_\Gamma, \Omega^1_{X_\Gamma})\to \rH^0(Y_\Gamma, \Omega^1_{Y_\Gamma})$
is an isomorphism, {\it i.e.}, 
\begin{equation}\label{toroidal}
q(Y_\Gamma)=q(X_\Gamma). 
\end{equation}
It is known that  $\pi^+\otimes \1^{\otimes d-1}$ (resp. $\pi^-\otimes \1^{\otimes d-1}$) contributes to $\rH^0(Y_\Gamma, \Omega^1_{Y_\Gamma})$ (resp.  $\rH^1(Y_\Gamma, \Omega^0_{Y_\Gamma})$). Since the latter two groups have the same dimension, it follows from \eqref{h-formula}  that 
 \begin{equation}\label{q-formula}
q(Y_\Gamma)= m(\pi^+\otimes \1^{\otimes d-1}, \Gamma)=
m(\pi^-\otimes \1^{\otimes d-1}, \Gamma).
\end{equation}

We will now focus on the case when $\Gamma$ is a congruence subgroup and switch  to the adelic setting which is better suited for computing  the irregularity.
For  any  open compact subgroup $K$ of $G(\A_{F,f})$, where $\A_{F,f}$ denotes the ring of finite adeles  of $F$, we consider the adelic quotient
\begin{equation}\label{quotient}
Y_{K} = G(F)\backslash G(\A_F)/K K_\infty.
\end{equation}

Let $G^1=\ker(\det: G\to M^1)$ be the derived group of $G$.  Since $G^1$ is simply connected and $G^1_\infty$ is non-compact,
$G^1(F)$ is dense in $G^1(\A_{F,f})$   by strong approximation  (see~\cite[Theorem 7.12]{platonov-rapinchuk}). It follows that the group of connected components of $Y_K$  is isomorphic to  the  idele class group:
 \begin{equation}\label{components}
 \pi_0(Y_K)\simeq \A_M^1/M^1 \det(K)M_\infty^1.
\end{equation}
To describe  each connected component of $Y_K$, choose  $t_i\in G(\A_{F,f})$, $1\leq i\leq h$, such that $(\det(t_i))_{1\leq i\leq h}$ forms a complete set of representatives of $\A_M^1/M^1 \det(K)M_\infty^1$, and let $\Gamma_i=G(F)\cap t_i K t_i^{-1} G_\infty$. Then
\begin{equation}\label{classical-adelic}
 G(F)\backslash G(\A_F)/K = \coprod_{i=1}^h \Gamma_i \backslash G_\infty
\text{  and }Y_{K} = \coprod_{i=1}^h Y_{\Gamma_i}.
\end{equation}

Therefore  \eqref{decomposition} and \eqref{q-formula}  can be rewritten as:
\begin{equation}\label{decomposition-adelic}
\rH^1(Y_K, \C) \, \simeq \,  \bigoplus_{\pi= \pi_\infty\otimes\pi_f}
\left(\rH^1(\Lie(G_\infty), K_\infty; \pi_\infty) \otimes \pi_f^K\right)^{\oplus m(\pi)} \text{  and}
\end{equation}
  \begin{equation}\label{q-formula-adelic}
q(Y_K)= \sum_{\substack{\pi= \pi_\infty\otimes\pi_f \\  \pi_\infty=\pi_{\iota}\otimes \1^{\otimes d-1}, \pi_{\iota} \simeq \pi^+}} m(\pi)\dim(\pi_f^K)=\sum_{\substack{\pi= \pi_\infty\otimes\pi_f \\  \pi_\infty=\pi_{\iota}\otimes \1^{\otimes d-1}, \pi_{\iota} \simeq \pi^-}} m(\pi)\dim(\pi_f^K),
 \end{equation}
where $\pi $ runs over all automorphic representation of $G(\A_F)$
occurring discretely,   with multiplicity $m(\pi)$,   in  $\mathrm{L}^2(G(F)\backslash G(\A_F))$.

\subsection{Irregularity growth}

Non-vanishing of $q(Y_\Gamma)$ for sufficiently small congruence subgroups is known by  a theorem  of    Shimura  \cite[Theorem 8.1]{shimura}, extending earlier works of Kazhdan and Borel-Wallach \cite[VIII]{borel-wallach}. Our Diophantine results require however the stronger assumption that   $q(Y_\Gamma)>n$, which we  establish as a corollary of
the next proposition.

\begin{prop}\label{growth}
For every open compact subgroup $K\subset G(\A_{F,f})$ such  that $q(Y_K)\neq 0$
there exist infinitely many primes $\fp$ of $F$ for which one can find 
 an explicit  finite index subgroup $K'\subset K$ differing form $K$ only at $\fp$,  such that
 $\pi_0(Y_{K'})=\pi_0(Y_K)$ and $q(Y_{K'})>q(Y_K)$.
\end{prop}

 \begin{proof}
 Since $q(Y_{K})\neq 0$ by assumption, formula \eqref{q-formula-adelic} implies that there exists an  automorphic representation $\pi$
with $\pi_\infty=\pi_{\iota}\otimes \1^{\otimes d-1}$, $\pi_{\iota} \simeq \pi^+$,  such that $m(\pi)\neq 0$   and  $\pi_f^K\neq 0$.

 Let $\fp$ be a prime of $F$ which splits  in $M$, so that $G(F_\fp)=\GL(n+1, F_\fp)$.
 Assume that  $K=K^0_\fp \times K^{(\fp)}$ where  $K^0_\fp= \GL(n+1, \fo_\fp)$ is the standard maximal compact subgroup of  $G(F_\fp)$ and $K^{(\fp)}$ is the part of $K$ away from $\fp$.
In particular $\pi_\fp$  is unramified. Moreover   $\pi_\fp$  is a unitary representation, since it
is a local component of an automorphic representation.
 By the main result of \cite{tadic}, $\pi_{\fp}$ is then the full induced representation of $\GL(n+1,F_\fp)$ from an unramified character $\mu$  of a parabolic subgroup
 $P(F_\fp)$.

 We claim that,  in our case,  $P$ is a proper parabolic subgroup. Otherwise $\pi_\fp$ will be  one dimensional, hence $G^1(F_\fp)$ will act trivially. Since by strong approximation  $G^1(F)G^1(F_\fp)$ is dense in $G^1(\A_F)$, the latter will act trivially on any smooth vector in $\pi$, contradicting  the fact that $\pi_{\iota} \not\simeq \1$.

Let     $\F_q=\fo/\fp$  be the residue field of $F_\fp$ and denote by
$P(\F_q)$ the corresponding  parabolic subgroup of $G(\F_q)$.  Let $K_{0,P}(\fp)$ is the inverse image of $P(\F_q)$ under the reduction modulo $\fp$ homomorphism
$\GL(n+1,\fo_\fp)\to \GL(n+1,\F_q)$.

Consider  $ K' = K_{0,P}(\fp)K^{(\fp)}$.
Since   $ \det(K')=\det(K)$,  \eqref{components} implies that   $\pi_0(Y_{K'})=\pi_0(Y_K)$.
 Moreover, formula \eqref{q-formula-adelic} implies that
$$q(Y_{K'})  \geq q(Y_{K}) + \dim(\pi_\fp^{K_{0,P}(\fp)})- \dim(\pi_\fp^{K^0_\fp})=q(Y_{K}) + \dim(\pi_\fp^{K_{0,P}(\fp)})- 1, $$
hence it suffices to show that the $K_{0,P}(\fp)$-invariants in $\pi_{\fp}$ form at least a $2$-dimensional space. We claim that we even have
\begin{equation}\label{split}
\dim(\pi_\fp^{K_{0,P}(\fp)})\geq n+1.
\end{equation}
Indeed, since $\mu$ is unramified, its restriction to $P\cap K^0_\fp$ is trivial. Therefore
by the Iwasawa decomposition $G(F_\fp)=P(F_\fp)\cdot K^0_\fp$ the restriction of
$\pi_\fp$ to $K^0_\fp$ is isomorphic to $ \Ind_{P(F_\fp)\cap K^0_\fp}^{K^0_\fp}(\1)$.

The subspace of $K_{0,P}(\fp)$-invariant vectors in $ \Ind_{P(F_\fp)\cap K^0_\fp}^{K^0_\fp}(\1)$ identifies naturally with the space of $\C$-valued functions on the set:
  $$
\left(P(F_{\fp})\cap K^0_\fp \right)\backslash K^0_\fp / K_{0,P}(\fp) \simeq  P(\F_q)\backslash  G(\F_q)/P(\F_q).
$$
 and the number of such double cosets is the number  of  double cosets of the Weyl group of $G$
relative to the  subgroup attached to $P$. We may assume, for getting a lower bound, that $P$ is maximal (and proper). The smallest number appears for $P$ of type $(n,1)$ and it is $n+1$. The claim follows.
\end{proof}

\begin{cor}\label{growth-cor}
For every arithmetic subgroup $\Gamma\subset G(F)$
 there exists an explicit   subgroup $\Gamma'$ of finite index  in  $\Gamma$ such that  $q(Y_{\Gamma'})>q(Y_\Gamma)$.
\end{cor}

\begin{proof}

By \cite[Theorem 8.1]{shimura} there exists an open compact subgroup $K$ of $G(\A_{F,f})$ such that $q(Y_{K})\neq 0$
(one might take  the principal congruence subgroup of level $2N$, which is included in the index $4$ subgroup of the
principal congruence subgroup of level $N$ considered by Shimura).
Denote by $h$ the cardinality of $\pi_0(Y_{K})$.
Applying recursively Proposition \ref{growth}  yields a finite index subgroup $K'\subset K$,  such that
 $\pi_0(Y_{K'})=\pi_0(Y_K)$ and $$q(Y_{K'})> h \cdot q(Y_\Gamma).$$

 Write
$ Y_{K'} = \coprod_{i=1}^h Y_{\Gamma'_i}$ as in  \eqref{classical-adelic},  and let  $\Gamma'=  \cap_{i=1}^h \Gamma'_i$. Since the irregularity cannot decrease by going to a finite cover, one has:
$$q(Y_{\Gamma'\cap \Gamma})\geq q(Y_{\Gamma'})\geq \underset{1\leq i\leq h}{\max} q(Y_{\Gamma'_i})\geq
\frac{1}{h} \sum_{i=1}^h q(Y_{\Gamma'_i})=\frac{1}{h} q(Y_{K'})>  q(Y_\Gamma).$$
\end{proof}

One can simplify the final step of the proof above  and use any  $\Gamma'_i$  instead of  $  \cap_{i=1}^h \Gamma'_i$, since Shimura's theory of canonical models implies that the  connected components of $Y_{K'}$  are all Galois conjugates,  hence share the same irregularity.

\section{Irregularity of arithmetic surfaces}

The positivity of $q(Y_\Gamma)$ is an essential ingredient in the proof of our Diophantine results.

The starting point for the arithmetic application of this paper was our knowledge that Rogawski's classification \cite{rogawski1,rogawski2}  of cohomological automorphic forms on $G$, combined with  some local representation theory, would allow us to compute
$q(Y_\Gamma)$ precisely and show that it does not vanish for some explicit  congruence subgroups $\Gamma$.
Marshall \cite{marshall} gives sharp asymptotic bounds for $q(Y_\Gamma)$ when $\Gamma$ shrinks, also by using Rogawski's theory.

In  this section we assume  that $n=2$.

\subsection{Rogawski's theory}

Rogawski \cite{rogawski1, rogawski2} gives  an explicit description, in terms of global Arthur packets,
 of the automorphic representations $\pi$ of $G(\A_F)$  occurring in \eqref{q-formula-adelic}, which we will now present.

Let $T$ denote the maximal torus of the standard upper-triangular Borel subgroup $B$   of $G$.

 Let $G'$ denote the quasi-split unitary group associated to $M/F$, so that $G$ is an inner form of $G'$.
 Note that  $G_v \simeq G_v'$ for any finite place $v$ and that  $G \simeq G'$ only for  $d=1$.

Let $\lambda$ be a unitary  Hecke character of $M$ whose restriction to $F$ is $\omega$,  and let $\nu$ be a unitary   character of  $\A_M^1/M^1$.

At a  place $v$ of $F$ which does not split in $M$, which includes any Archimedean $v$, 
the local Arthur packet $\Pi'(\lambda_v,\nu_v)$
consists of a square-integrable representation $\pi_s(\lambda_v,\nu_v)$ and a non-tempered representation $\pi_n(\lambda_v,\nu_v)$ of $G'(F_v)$. These constituents of the packet
can be described  (see   \cite[\S12.2]{rogawski1}) as the unique subrepresentation and the corresponding (Langlands) quotient representation of  the induction of the character of $B(F_v)$  which is trivial on the unipotent subgroup  and given on   $T(F_v)$  by:
\begin{equation}\label{torus-char}
(\bar\alpha,\beta,\alpha^{-1})\mapsto \lambda_v(\bar\alpha)|\alpha|_{M_v}^{3/2} \nu_v(\beta)\text{, where } 
\alpha\in  M_v^\times, \beta\in M_v^1.
\end{equation}
If one considers  unitary induction, then one has to divide the above character by  the square root of the modular character
of  $B(F_v)$, that is to say by $(\bar\alpha,\beta,\alpha^{-1})\mapsto |\alpha|_{M_v}$.

 At any  finite place $v$ of $F$ which splits in $M$, $G_v\simeq G'_v$ also splits and is isomorphic to
 $\GL(3,F_v)$. The  local Arthur packet
 $\Pi'(\lambda_v,\nu_v)$ has a unique element  $\pi_n(\lambda_v,\nu_v)$ which is  induced from the  character:
 $$\left (\begin{array}{c|@{}c} h_2 & \begin{array}{c} \ast \\ \ast \end{array}  \\ \hline
 \begin{array}{cc} 0 & 0 \end{array}   & h_1
\end{array}\right)\mapsto \lambda_v(\det(h_2)) |\det(h_2)|_v^{3/2} \nu_v(h_1)$$
 of the maximal parabolic of type $(2,1)$ in $\GL(3,F_v)$ (see~\cite[\S1]{rogawski2}).

For almost all $v$, $\pi_n(\lambda_v,\nu_v)$ is necessarily unramified. We set
$$\Pi'(\lambda,\nu)=\left\{\otimes_v \pi_v | \pi_v\in \Pi'(\lambda_v,\nu_v) \text{ for all } v \text{ , and }  \pi_v\simeq\pi_n(\lambda_v,\nu_v) \text{ for almost all } v\right\}.$$

Recall that a CM type $\Phi$  on $M$ is the choice, for each Archimedean place $v$ of $F$, of
an isomorphism $M\otimes_{F,v}\R\simeq \C$.

\begin{defn} \label{CM-type}
Let $\Xi$  denote the set of  pairs $(\lambda,\nu)$ where $\lambda$ is a unitary  Hecke character of $M$ whose restriction to $F$ is $\omega$,  and  $\nu$ is a unitary   character of  $\A_M^1/M^1$, such that 
\begin{equation}\label{weight1}
\lambda_\infty(z)=\prod_{v\in \Phi} \frac{\bar{z}_v}{|z_v|}\text { , for all } z\in M_\infty\,
\text { and } \\
\end{equation}
$$ \nu_\infty(z)=\prod_{v\in \Phi} z_v\text { , for all } z\in M_\infty^1, $$
for some CM type $\Phi$ on $M$. 
\end{defn}

\begin{thm}[Rogawski \cite{rogawski1, rogawski2}]\label{thm-rogawski}
\begin{enumerate}
\item For every  $(\lambda,\nu)\in \Xi$, $\Pi'(\lambda,\nu)$ is a global Arthur packet for $G'$ such that for all infinite  $v$,  $\pi_n(\lambda_v,\nu_v)=\pi^+$ or $\pi^-$.
\item $\Pi'(\lambda,\nu)$ can be transferred to an Arthur packet $\Pi(\lambda,\nu)$ on $G$ such that
$\Pi(\lambda_v,\nu_v)=\{\1\}$ at all the Archimedean places $v\neq \iota$, and
$\Pi(\lambda_v,\nu_v)=\Pi'(\lambda_v,\nu_v)$ at the remaining places.
\item Denote by $W(\lambda\nu_M)\in \{\pm 1\}$  the root number of 
Hecke character  $\lambda\nu_M$, where  $\nu_M(z)= \nu(\bar z/z)$ for $z\in \A_M^\times$, and by
 $s(\pi)$  the number of finite  places $v$ such that $\pi_v\simeq \pi_s(\lambda_v,\nu_v)$. Then
$$\pi \in \Pi(\lambda,\nu)\text { is automorphic if and only if }W(\lambda\nu_M)=(-1)^{d-1+s(\pi)}.$$
Moreover, in this case the global multiplicity $m(\pi)$ is $1$.
\item Any automorphic representation $\pi$ of $G(\A_F)$  such that  $\pi_{\iota} \simeq \pi^\pm$
and  $\pi_v=\1$ at all the Archimedean places $v\neq \iota$, belongs to
$\Pi(\lambda,\nu)$ for some  $(\lambda,\nu)\in \Xi$.
\end{enumerate}
\end{thm}

\begin{proof} Let $H= {\rm U}(2)\times {\rm U}(1)$ be the unique elliptic endoscopic group, shared by $G'$ and all its inner forms over $F$.
The embedding of $L$-groups ${}^LH\hookrightarrow {}^LG ={}^LG'$   depends on the choice of a
Hecke character $\mu$ of $M$, whose restriction to $F$ is $\omega$, and
 allows one to transfer discrete $L$-packets on $H$ to automorphic $L$-packets on $G$ (see \cite[\S13.3]{rogawski2}).
  The character $\mu$ being fixed, any pair of  characters $(\lambda,\nu)\in \Xi$ uniquely determines a (one-dimensional) character of $H$, whose endoscopic  transfer is $\Pi'(\lambda,\nu)$
  (see \cite[\S1]{rogawski2}).

Denote by $W_F$ (resp. $W_M$) the  global Weil group of $F$ (resp. $M$).
By {\it loc.cit.}, the restriction  to $W_M$ of the global Arthur parameter
$$W_F\times {\rm SL}(2, \C)\to \, {}^LG = \GL(3, \C) \rtimes \mathrm{Gal}(M/F)$$
 of $\Pi'(\lambda,\nu)$ is given by the $3$-dimensional representation
$(\lambda \otimes \mathrm{St}) \oplus (\nu_M \otimes \1)$, where $\mathrm{St}$ (resp. $\1$) is the
standard $2$-dimensional (resp. trivial) representation of ${\rm SL}(2, \C)$.
By \cite[p.62]{langlands} the restrictions to $\C^\times$ of the Langlands parameters of $\pi^+$ and $\pi^-$ are given by
$z \mapsto \left(\begin{smallmatrix}\bar{z}& 0 & 0\\0 & z/{\bar z}&0\\0 & 0 &  z^{-1} \end{smallmatrix}\right)$
and its complex conjugate, hence  for every Archimedean place $v$ one has  $\pi_n(\lambda_v,\nu_v)=\pi^+$ or $\pi^-$,
depending on the choice of isomorphism $M\otimes_{F,v}\R\simeq \C$ in the CM type $\Phi$.
 It follows that for every Archimedean place $v$, $\Pi'(\lambda_v,\nu_v)$ is a packet containing a discrete series representation of $G'_v$, and thus by  \cite[\S 14.4]{rogawski1},  there will be a corresponding Arthur packet $\Pi(\lambda,\nu)$ of representations of $G(\A_F)$ such that at any Archimedean place $v\neq\iota$,  $\Pi(\lambda_v,\nu_v)$ is a singleton consisting of a finite-dimensional representation of the compact real group $G(F_v)={\rm U}(3)$.
In the notation of \cite[p.397]{rogawski2} the representations $\pi^+$ and $\pi^-$ have parameters
 $(r,s)=(1,-1)$ and   $(r,s)=(0,1)$, respectively,  and hence,   by the recipe on the same page,
 the highest weight of the associated finite-dimensional representation equals $(1,0,-1)$.
  Therefore at every   Archimedean $v\neq \iota$ we have  $\Pi(\lambda_v,\nu_v)=\{\1\}$.

So far we have established (i) and (ii), while (iii) is the content of  \cite[Theorem 1.1]{rogawski2}.

Conversely, any $\pi$ as in  (iv) is discrete, hence belongs to an Arthur packet $\Pi$ on $G$,
which can be transferred to an Arthur packet $\Pi'$ on $G'$ (see \cite[\S 14.4 and  Proposition 14.6.2]{rogawski1}).  By definition,  $\Pi_v=\Pi'_v$ at $v=\iota$ and at all the finite places $v$ (where, as noted earlier, $G_v=G'_v$). In particular  $\pi^+$ or $\pi^-$ belongs to $\Pi_\iota=\Pi'_\iota$, hence $\Pi'$ arises by endoscopy from $H$, that is to say equals $\Pi'(\lambda,\nu)$ for some  unitary  Hecke  character $\lambda$  of $M$ whose restriction to $F$ is $\omega$, and some  unitary  character of  $\A_M^1/M^1$ (see \cite[Theorem 13.3.6]{rogawski1}).  Since $\Pi_v=\{\1\}$  for all the Archimedean places $v\neq \iota$, by the above mentioned recipe $\Pi(\lambda_v,\nu_v)$ contains either $\pi^+$ or $\pi^-$, implying that $(\lambda,\nu)\in \Xi$
(see Definition \ref{CM-type}).
\end{proof}

\subsection{Irregularity of the connected components}

Let $ K$ be an open compact subgroup of $G(\A_{F,f})$.
Using   Theorem \ref{thm-rogawski}(iii) one can  transform    \eqref{q-formula-adelic} into the formula:
  \begin{equation}\label{q-formula-bis}
4q(Y_K)= \sum_{(\lambda,\nu)\in \Xi}
\sum_{ \pi \in \Pi(\lambda,\nu)} \dim(\pi_f^K)(W(\lambda\nu_M)+(-1)^{d-1+s(\pi)}).
 \end{equation}

We will now deduce a similar formula for the irregularity of the connected component of identity
$Y_{\Gamma}$ of $Y_K$, where $\Gamma = G(F)\cap K  G_\infty$.

Recall (see \eqref{components}) that $\pi_0(Y_K)\simeq  \A_M^1/M^1 \det(K)M_\infty^1$, and denote by $\widehat {\pi_0(Y_K)}$
its  (finite, abelian) group of characters.  Consider the free action of $\widehat {\pi_0(Y_K)}$ on  the set $\Xi$ given by $\left(\chi, (\lambda,\nu)\right) \mapsto (\lambda\chi_M^{-1},\nu\chi)$,  and denote by $ \Xi/\widehat{\pi_0(Y_K)}$ the quotient set. Since for any  $\pi \in \Pi(\lambda,\nu)$ and   any  $\chi\in \widehat {\pi_0(Y_K)}$  one has $\pi\otimes\chi\in \Pi(\lambda\chi_M^{-1},\nu\chi)$, the group
$\widehat{\pi_0(Y_K)}$ acts freely on the set of automorphic representations contributing to $q(Y_K)$.
Moreover this action preserves $\lambda\nu_M$, $s(\pi)$ and the dimension of $\pi_f^K$.
Hence, in the notations of \eqref{classical-adelic}, for any $1\leq i\leq h$, the image of the composite map
$$\rH^1(\Lie(G_\infty), K_\infty; \pi_\infty)\otimes \pi_f^K\to \rH^1(Y_K, \C) \to \rH^1(Y_{\Gamma_i}, \C), $$
 where the first map comes from  \eqref{decomposition-adelic} and the second  from the inclusion 
 $Y_{\Gamma_i}\subset Y_K$,
does not change when replacing $\pi$ by $\pi\otimes\chi$ for any $\chi\in\widehat {\pi_0(Y_K)}$.

It follows that   $q(Y_{\Gamma_i})\leq \frac{1}{h} q(Y_K)$ for  all $1\leq i\leq h$.
Since  $\sum_{i=1}^h q(Y_{\Gamma_i})= q(Y_K)$, we deduce that  $q(Y_{\Gamma_i})= \frac{1}{h} q(Y_K)$
for all $1\leq i\leq h$. This establishes  the   formula:
 \begin{equation}\label{q-formula-ter}
4q(Y_\Gamma)= \sum_{(\lambda,\nu)\in \Xi/\widehat{\pi_0(Y_K)}}
\sum_{ \pi \in \Pi(\lambda,\nu)} \dim(\pi_f^K)(W(\lambda\nu_M)+(-1)^{d-1+s(\pi)}).
 \end{equation}

This formula shows the importance of calculating $\dim(\pi_f^K)$ which, when 
$K$ is of the form $\prod_v K_v$ with $v$ running over all the finite places of $F$, 
can be reduced to a local computation of $\dim(\pi_v^{K_v})$. This will be taken up in the following section  
at  places $v$ where $K_v$ is not the hyperspecial maximal compact subgroup $K_v^0$. 

\subsection{Levels of induced representations}

Let $\fp$ be a prime of $F$ divisible by a unique prime $\fP$ of
$M$ and let $\F_q$ be the residue field $\fo/\fp$.
In this section we exhibit open compact subgroups $K_\fp$ of   $G(F_{\fp})$ for which $\pi_n(\lambda_\fp,\nu_\fp)$ (resp. $\pi_s(\lambda_\fp,\nu_\fp)$) admit a non-zero $K_\fp$-invariant subspace, and compute in some cases the exact dimension of this space.

For every  integer $m\geq 1$,  we define  the open compact subgroup $K(\fP^m)$ (resp.  $K_0(\fP^m)$, resp.  $K_1(\fP^m)$) of $G(F_{\fp})$ as the kernel (resp. the inverse image of upper triangular, resp.  upper  unipotent, matrices) of the composite homomorphism:
\begin{equation}
G(\fo_{\fp}) \hookrightarrow {\rm GL}(3, \fO_{\fP} ) \to {\rm GL}(3, \fO/\fP^m).
\end{equation}

\begin{lemma} \label{iwahori}
Let  $m\in \Z_{>0}$ be such that the character \eqref{torus-char} is trivial on $K_1(\fP^m)\cap T(F_\fp)$.  Then both $\pi_n(\lambda_\fp,\nu_\fp)$  and  $\pi_s(\lambda_\fp,\nu_\fp)$ have non-zero fixed vectors under  $K_1(\fP^m)$.
\end{lemma}
\begin{proof}
Let $J$ denote the Jacquet functor sending admissible $G(F_{\fp})$-representations to
admissible $T(F_{\fp})$-representations. The Jacquet functor is exact and  its basic properties imply:
\begin{equation}\label{jacquet}
\begin{split}
J(\pi_s(\lambda_\fp,\nu_\fp)) :  (\bar\alpha,\beta,\alpha^{-1})&\mapsto \lambda_{\fp}(\bar\alpha)\nu_{\fp}(\beta)|\alpha|_{M_{\fp}}^{3/2}=\lambda_{\fp}(\bar\alpha)\nu_{\fp}(\beta)|\alpha|_{M_{\fp}}^{1/2}\cdot |\alpha|_{M_{\fp}},\\
J(\pi_n(\lambda_\fp,\nu_\fp)) : (\bar\alpha,\beta,\alpha^{-1})&\mapsto \lambda_{\fp}(\bar\alpha)\nu_p(\beta)|\alpha|_{M_{\fp}}^{1/2}=\lambda_{\fp}(\alpha^{-1})\nu_{\fp}(\beta)|\alpha|_{M_{\fp}}^{-1/2}\cdot |\alpha|_{M_{\fp}}.
\end{split}
\end{equation}

One knows that $K_1(\fP^m)$  admits an Iwahori decomposition:
$$K_1(\fP^m)= (K_1(\fP^m)\cap N(F_{\fp})) \cdot (K_1(\fP^m)\cap T(F_{\fp})) \cdot
(K_1(\fP^m)\cap \bar N(F_{\fp})),$$
where $N(F_{\fp})$ (resp. $\bar N(F_{\fp})$) denotes the unipotent of the standard (resp. opposite) Borel
containing $T(F_{\fp})$.
This is proved for the principal congruence subgroup $K(\fP^m)$ in \cite[Proposition 1.4.4]{casselman} and the extension to $K_1(\fP^m)$ is straightforward. Now by the proof of \cite[Proposition 3.3.6]{casselman}, given any admissible $G(F_{\fp})$-representation $V$, one has a canonical surjection:
$$V^{K_1(\fP^m)}\twoheadrightarrow J(V)^{K_1(\fP^m)\cap T(F_{\fp})}.$$
Since both characters in \eqref{jacquet} are trivial on $K_1(\fP^m)\cap T(F_{\fp})$, the claim follows.
\end{proof}

\begin{lemma} \label{inert}
Suppose that $\fp$ is inert in $M$ and that $(\lambda_\fp,\nu_\fp)$ is unramified.
Then the dimension of the $K_0(\fp)$-fixed subspace of  both  $\pi_s(\lambda_\fp,\nu_\fp)$ and $\pi_n(\lambda_\fp,\nu_\fp)$ is $1$. Moreover the dimension
of the $K(\fp)$-fixed subspace of  $\pi_s(\lambda_\fp,\nu_\fp)$ (resp. $\pi_n(\lambda_\fp,\nu_\fp)$) is $q^3$ (resp. $1$).
\end{lemma}

\begin{proof} Since $(\lambda_\fp,\nu_\fp)$ is unramified,  restriction to the standard hyperspecial maximal compact subgroup $K^0_\fp$ of $G(F_{\fp})$ yields,  by Iwasawa decomposition
$G(F_\fp)=B(F_\fp)\cdot K^0_\fp$, the following exact sequence:
$$0\to \pi_s(\lambda_\fp,\nu_\fp)_{|K^0_\fp}\to \Ind_{B(F_{\fp}) \cap K^0_\fp}^{K^0_\fp}(\1) \to
\pi_n(\lambda_\fp,\nu_\fp)_{|K^0_\fp}\to 0. $$

The subspace of $K(\fp)$-invariant vectors in $\Ind_{B(F_{\fp}) \cap K^0_\fp}^{K^0_\fp}(\1)$ identifies naturally with the space of $\C$-valued functions on the set:
  $$
\left(B(F_{\fp})\cap K^0_\fp \right)\backslash K^0_\fp / K(\fp)  \simeq  B(\F_q)\backslash  G(\F_q),
$$
 on which $K^0_\fp/K(\fp) =  G(\F_q)$  acts by right translation.
 By  the Iwahori  decomposition,  since $G(\F_q)$ has rank $1$,
 the representation $\Ind_{B(\F_q)}^{G(\F_q)}(\1)$
    has exactly two irreducible constituents which are the trivial representation and the Steinberg representation, implying that both  $\pi_n(\lambda_\fp,\nu_\fp)^{K_0(\fp)}$  and  $\pi_s(\lambda_\fp,\nu_\fp)^{K_0(\fp)}$  are one-dimensional. Since $\pi_s(\lambda_\fp,\nu_\fp)^{K^0_\fp}=0$,  it follows that
 $\pi_n(\lambda_\fp,\nu_\fp)^{K(\fp)}$  (resp. $\pi_s(\lambda_\fp,\nu_\fp)^{K(\fp)}$ ) is isomorphic to the trivial (resp.  Steinberg) representation of $G(\F_q)$, hence its dimension equals $1$
(resp.  $q^3$).
\end{proof}

\subsection{Surfaces with positive irregularity}

The existence of  Hecke characters $\lambda$ of $M$ satisfying \eqref{weight1} goes back to  Chevalley and Weil. We will show that there are still such characters  if one  further imposes their restriction to $F$
to be $\omega$.

\begin{lemma}\label{CM-Hecke}
For any CM extension $M/F$ and any CM type $\Phi$ on $M$, there exist  Hecke characters $\lambda$ of $M$ whose restriction to $F$ equals $\omega$, such that $\displaystyle \lambda_\infty(z)=\prod_{v\in \Phi} \frac{\bar{z}_v}{|z_v|}$ for all $z\in M_\infty^\times$. 
\end{lemma}
\begin{proof}
Since $M$ is totally imaginary, $\lambda_\infty$ and $\omega$ agree on
$F_\infty^\times$, hence there is a character $\lambda_0$ of $\A_F^\times M_\infty^\times$ extending both.

We show now that $\lambda_0$ can be extended to a Hecke character $\lambda$ of $M$, which will obviously satisfy the assumptions of the lemma. Since $M/F$ is totally imaginary,  $\fo^{\times 2}$ has finite index  in $\fO^{\times}$. By \cite[Th\'eor\`eme 1]{chevalley} there exists an open compact subgroup $U$ of $\A_{M,f}^\times$ such that $U\cap \fO^{\times}\subset \fo^{\times 2}$. We may, and we will,  assume that $U$ is contained in the congruence subgroup whose level is the relative different of $M/F$ and by replacing $U$ by $U\cap\bar U$ we can further assume that $U=\bar U$. 
Since the Artin conductor of $\omega$ is the relative discriminant of $M/F$, it follows that $\omega$ is
trivial on $U\cap \A_{F,f}^\times$.  Hence one can extend $\lambda_0$ to a 
character of $\A_F^\times U M_\infty^\times$ by letting it be trivial on $U$. 

Suppose we knew that
\begin{equation}\label{intersection}
M^{\times}\cap \A_F^\times U M_\infty^\times=F^\times.
\end{equation}
Then there is a unique character of $M^{\times} \A_F^\times U M_\infty^\times$ extending both
$\lambda_0$ and the trivial character of $M^{\times} U$. Since $\A_{M}^{\times}/M^{\times}\A_F^\times U M_\infty^\times$ is a finite abelian (idele class) group, the  character above can be further extended to
a   character $\lambda$ of $\A_{M}^{\times}/M^{\times}$, and  any such extension has the desired properties.

 It remains to prove \eqref{intersection}. Let $x\in M^{\times}\cap \A_F^\times U M_\infty^\times$.
 Then
 $$\bar x/x\in M^1\cap U M_\infty^1= \fO^{\times} \cap U M_\infty^1\subset \fo^{\times 2}.$$
 Since  $F^{\times 2} \cap M^1=\{1\}$, we have  $x=\bar x\in F^\times$.
\end{proof}

\begin{prop}\label{positivity} Fix any  Hecke character $\lambda$ of $M$ 
satisfying \eqref{weight1} whose restriction to $F$ is $\omega$.
 Let   $\fp$ be a prime of $F$ which  splits in   $M$ and is relatively prime to the conductor $\fC$ of 
 $\lambda$.  If $W(\lambda^{3})=(-1)^{d}$ we choose a prime $\fq$  of $F$ which does not split in  $M$; if not, we take $\fq=\fo$.
Then $q(Y_{\Gamma_1(\fC)\cap \Gamma_0(\fp\fq)})>2$.

\end{prop}

\begin{proof}
Let $K= K_1(\fC)\cap K_0(\fp\fq)$, so that  $\Gamma=G(F)\cap K K_\infty$.
Let $\Pi(\lambda,\lambda^{-1}_{|M^1})$ be the global Arthur packet on $G$ associated to
$\lambda$. Let  $\pi=\otimes_v \pi_v\in \Pi(\lambda,\lambda^{-1}_{|M^1})$ be such that $\pi_\iota=\pi^+$, $\pi_v=\1$ for every infinite place $v\neq \iota$, $\pi_v=\pi_{n,v}$ for every finite $v\neq \fq$, and finally if $W(\lambda^{3})=(-1)^{d}$ then $\pi_{\fq}=\pi_{s,\fq}$.
By Theorem \ref{thm-rogawski}(iii), $\pi$ is automorphic, and by
 Lemma \ref{iwahori}, we have $\pi_{n,v}^{K_{v}}\neq 0$ for all finite places $v\neq \fp, \fq$.
Moreover, if $\fq\neq \fo$, then $\pi_{s,\fq}^{K_\fq}\neq 0$ by Lemma \ref{iwahori}
(resp. Lemma  \ref{inert} ) 
 if $\fq$ divides (resp.  does not divide) $\fC$.
Finally  \eqref{split}  implies that $\dim(\pi_{\fp}^{K_0(\fp)}) \geq 3$, hence $q(Y_{\Gamma})\geq 3$
by \eqref{q-formula-ter}  as claimed. \end{proof}

\begin{rem}
Since restriction of $\lambda$ to $\A_F^\times$ equals $\omega$, its conductor $\fC$ is divisible by the different of $M/F$.
Hence, unless $M/F$ is unramified everywhere, one might take as $\fq$ a place where
$M/F$ is ramified and  Proposition \ref{positivity} applies then to $K= K_1(\fC)\cap K_0(\fp)$. 
Given a totally real number field $F$, there exists a  totally imaginary  quadratic extension $M/F$ unramified everywhere if and only if all the units in $F$ have norm 1.
\end{rem}

For the rest of this section we assume that  $F=\Q$, so that $G$ is quasi-split.

\begin{prop} \label{low}
Any $\Gamma$ as in Theorem \ref{picard-open} is neat and $q(Y_\Gamma)>2$.
\end{prop}

\begin{proof}
There exists an open compact subgroup $ K$ of $ G(\A_{\Q,f})$ such that
$\Gamma=G(\Q)\cap  K G(\R) $.

We claim that there exists a  Hecke character  $\lambda$ of $M$ of conductor $\fD$
 satisfying \eqref{weight1}, whose restriction to $F$ equals $\omega$. 
If  $D\neq 3$ is odd or if $8$ divides $D$ such characters, called canonical,  are proved to exists by 
Rohrlich \cite{rohrlich}. If $D>4$ is even but not divisible by $8$, then by Yang (see \cite[p.88]{yang})  there are  such characters, called the {\it simplest}. Finally for $D=3$ (resp. D=4) the claim follows from the existence of a CM elliptic curve over $\Q$ of conductor $27$ (resp. $32$).

By definition,  $(\lambda,\lambda^{-1}_{|M^1})\in \Xi$ and is trivial on
$K_1(\fD)\cap T(\A_{\Q,f})$.  Then Lemma \ref{iwahori} implies that:
\begin{equation}\label{inv-non-zero}
\pi_f^{K_1(\fD)}\neq 0\text { , for all } \pi \in \Pi(\lambda,\lambda^{-1}_{|M^1}).
\end{equation}

\medskip
In case (ii) of Theorem \ref{picard-open} where $\Gamma=\Gamma(N)\cap \Gamma_1(\fD)$
we fix a prime $p$ dividing $D$ and   $\pi=\otimes_v \pi_v\in  \Pi(\lambda,\lambda^{-1}_{|M^1})$ such that
$\pi_v=\pi_n(\lambda_v,\lambda^{-1}_{|M_v^1})$ for all  $v\neq p,N$,  $\pi_N=\pi_s(\lambda_N,\lambda^{-1}_{|M_N^1})$ and
 $$\pi_p=\begin{cases} \pi_n(\lambda_p,\lambda^{-1}_{|M_p^1})\text{ , if } W(\lambda^3)=-1,\\
  \pi_s(\lambda_p,\lambda^{-1}_{|M_p^1})\text{ , if } W(\lambda^3)=1.\end{cases}$$
Since $\Gamma(N)$ is neat by Lemma \ref{neat}, we can apply \eqref{q-formula-ter} which, when combined with
Lemma \ref{inert}, yields
$$q(Y_{\Gamma(N)\cap\Gamma_1(\fD)})\geq \dim(\pi_N^{K(N)})\geq N^3\geq 3.$$

We now turn  to case (i) and suppose first that  $M$ has  class  number $h \geq 3$.
 For any class character $\xi$ one has $(\lambda\xi,\lambda^{-1}_{|M^1})\in \Xi$
 giving $h$ pairwise distinct elements in $\Xi/\widehat {\pi_0(Y_{K_1(\fD)})}$.
  Fix a prime $p$ dividing $D$ and consider
  $\pi=\otimes_v \pi_v\in  \Pi(\lambda\xi,\lambda^{-1}_{|M^1})$ such that
$\pi_v=\pi_n(\lambda_v\xi_v,\lambda^{-1}_{|M_v^1})$ for all  $v\neq p$ and
 $$\pi_p=\begin{cases} \pi_n(\lambda_p\xi_p,\lambda^{-1}_{|M_p^1})\text{ , if } W(\lambda^3)=1,\\
  \pi_s(\lambda_p\xi_p,\lambda^{-1}_{|M_p^1})\text{ , if } W(\lambda^3)=-1.\end{cases}$$
Since $\Gamma_1(\fD)$ is neat by Lemma \ref{neatbis}, one can apply \eqref{q-formula-ter} which combined with
 \eqref{inv-non-zero} yields  $q(Y_{\Gamma_1(\fD)})\geq h\geq 3$.

If $M$ is one of the $18$ imaginary quadratic fields of class number $2$, then its
fundamental discriminant $D$ has (exactly) two distinct prime divisors $p<q$.
For each  character $\lambda$ on $M$ as above,
consider  $\pi\in\Pi(\lambda,\lambda^{-1}_{|M^1})$
such that $\pi_v=\pi_n(\lambda_v,\lambda^{-1}_{|M_v^1})$ for all  $v\neq p,q$ and
$$(\pi_p,\pi_q)=\begin{cases} (\pi_n(\lambda_p,\lambda^{-1}_{|M_p^1}), \pi_n(\lambda_q,\lambda^{-1}_{|M_q^1}))
\text{ or } (\pi_s(\lambda_p,\lambda^{-1}_{|M_p^1}), \pi_s(\lambda_q,\lambda^{-1}_{|M_q^1}))
\text{ , if } W(\lambda^3)=1,\\
(\pi_n(\lambda_p,\lambda^{-1}_{|M_p^1}), \pi_s(\lambda_q,\lambda^{-1}_{|M_q^1}))
\text{ or } (\pi_s(\lambda_p,\lambda^{-1}_{|M_p^1}), \pi_n(\lambda_q,\lambda^{-1}_{|M_q^1}))
\text{ , if } W(\lambda^3)=-1.\end{cases}$$
If $D\neq 24$ then $\Gamma_1(\fD)$ is neat by Lemma \ref{neatbis} and
\eqref{q-formula-ter} implies that $q(Y_{\Gamma_1(\fD)})\geq 2\cdot 2 =4$.
If $D=24$ then $\Gamma(\fD)$ is neat by Lemma  \ref{neat}, since $4$ divides $\fD$, and again
$q(Y_{\Gamma(\fD)})\geq 4$.

Finally,  we consider  the nine imaginary quadratic fields of class number $1$.

For $D\in  \{7,11,19, 43, 67, 163\}$ there is a unique  character 
$\lambda$ as in the beginning of the proof. 
Any character of $(1+\sqrt{-D}\fO/ 1+D\fO)\simeq \Z/D\Z$ lifts to a finite order Hecke character $\xi$ of $M$ with trivial restriction to $\Q$, hence $(\lambda\xi, \lambda^{-1}_{|M^1})\in \Xi$.
Let  $\pi=\otimes_v \pi_v\in  \Pi(\lambda\xi, \lambda^{-1}_{|M^1})$ be such that
$\pi_v=\pi_n(\lambda_v\xi_v,\lambda^{-1}_{|M_v^1})$ for all  $v\neq D$ and
 $$\pi_D=\begin{cases} \pi_n(\lambda_D\xi_D,\lambda^{-1}_{|M_D^1})\text{ , if } W(\lambda^3)=1,\\
  \pi_s(\lambda_D\xi_D,\lambda^{-1}_{|M_D^1})\text{ , if } W(\lambda^3)=-1.\end{cases}$$
Since $\Gamma(D)$ is neat by Lemma \ref{neat}, by \eqref{q-formula-ter} we get
$q(Y_{\Gamma(D)})\geq D\cdot \dim(\pi_D^{K(D)})\geq D$.

For $D=3$ the same argument with $D^2$ instead of $D$, implies that
$q(Y_{\Gamma(9\fO)})\geq 3$.

For $D=4$ (resp. $D=8$) the group $\Gamma(8\fO)$ (resp. $\Gamma(2\sqrt{-8}\fO)$) is neat by Lemma  \ref{neat}
and it is an  exercise on idele class groups to show that there are at least three Hecke characters of $M$ 
satisfying \eqref{weight1} whose restriction to $\Q$ is  $\omega$,
and whose conductor divides $8$ (resp. $2\sqrt{-8}$).
It follows then from \eqref{q-formula-ter} and \eqref{inv-non-zero} that for $D=4$ (resp. $D=8$) one has
$q(Y_{\Gamma(8\fO)})\geq 3$ (resp. $q(Y_{\Gamma(2\sqrt{-8}\fO)})\geq 3$).
 \end{proof}

\begin{rem}
The computation of the smallest level $K$ for which there exists an  automorphic representation $\pi\in\Pi(\lambda, \nu)$ such that $\pi_f^K\neq 0$ is analyzed in detail in \cite{dimitrov-ramakrishnan}. In particular, if $\lambda$ is a canonical character, we check that the level subgroup at any  $p$ dividing  $D$ is precisely the one conjectured by B.~Gross, namely the index $2$ subgroup of the maximal parahoric subgroup with reductive quotient $\mathrm{PGL}(2)$.
\end{rem}

\section{The Albanese map and Mordellicity}\label{mordel}

A major ingredient in the proof of our theorems is the Mordell-Lang conjecture for abelian varieties in characteristic zero, established by
Faltings  \cite{faltings} using some  earlier work of himself \cite{faltings-annals} and Vojta \cite{vojta} (see  Mazur's detailed account \cite{mazur}).

\begin{thm}[Mordell-Lang conjecture : theorem of Faltings]\label{thmF}
Suppose $A$ is an abelian variety over $\C$ and  $Z \subset A$ a closed subvariety. Then for any finitely generated field extension $k$  of $\Q$
over which $Z\subset A$ is defined, the set $Z(k)$ is contained in a  union of  finitely many translates of abelian subvarieties of  $A$,  each of which is defined over $k$ and  contained in $Z$.
\end{thm}

The following corollary was proved in Moriwaki \cite[Theorem 1.1]{moriwaki}. He stated it  for number fields,  but the proof is the same for finitely generated fields over $\Q$. 
\begin{cor}\label{brodyan}
Let   $X$ be a connected smooth projective variety over $\C$ which does not admit a dominant map to its Albanese variety. Then for
any finitely generated field extension $k$  of $\Q$ over which $X$ is defined, the set $X(k)$ is not Zariski dense in $X$.
\end{cor}

\begin{proof} The conclusion is obvious if $X(k)$ is empty so we may choose  a point of
 $X(k)$ to define the  Albanese map over $k$:
$$j: X\to \Alb(X).$$
Applying Theorem \ref{thmF} to the closed subvariety $Z= j(X)$ of  $\Alb(X)$ we get a finite number, say $m\geq 1$, of translates $Z_i$ of
abelian subvarieties of $\Alb(X)$ defined over $k$ and  such that
$$Z(k) \, \subset \, \bigcup_{i=1}^m \, Z_i(k) \, \, \text { and } \, \, Z_i \subset Z.$$

Since $j$  is defined over $k$, each   $k$-rational point of $X$  is contained in  $j^{-1}(Z_i)$ for some $i$. If  $j^{-1}(Z_i)$ were not a proper closed subvariety of  $X$,  the universal property of the Albanese map would imply that   $Z=Z_i=\Alb(X)$,
contradicting the assumption that   $j$ is not dominant.
\end{proof}

\begin{prop}\label{bombieri-lang}
For every arithmetic subgroup $\Gamma\subset G(F)$  there exists a  finite cover of $Y_{\Gamma}$ 
whose points over any finitely generated field extension of $\Q$ are not Zariski dense, {\it i.e.}, 
the Bombieri-Lang conjecture holds for that cover.  
\end{prop}

\begin{proof}
Applying  Corollary \ref{growth-cor} recursively yields a  finite index  subgroup $\Gamma'\subset \Gamma$, which one can assume to be torsion free,
such that $q(Y_{\Gamma'})>n=\dim(Y_{\Gamma'})$. It suffices then to apply Corollary \ref{brodyan}   to $Y_{\Gamma'}$ in the compact case and, in view of \eqref{toroidal}, 
to $X_{\Gamma'}$ in the non-compact case.
\end{proof}

\begin{proof}[Proof of Theorem \ref{picard-compact}]
By Proposition \ref{positivity}, we have  $q(Y_{\Gamma})>2$, hence  $Y_{\Gamma}$ does not admit a dominant map to its Albanese variety.
Moreover $Y_{\Gamma}$ is a geometrically irreducible smooth projective  surface, hence
by Corollary \ref{brodyan}  $Y_{\Gamma}(k)$ is not Zariski dense in $Y_{\Gamma}$ for any finitely generated field extension $k$ of $\Q$ over which
$Y_{\Gamma}$ is defined. If  $Y_{\Gamma}(k)$ were infinite, then $Y_{\Gamma}$ would contain  an irreducible curve $C$ defined over $k$ and such that   $C(k)$  infinite. Since $C(k)$ is Zariski dense in $C$, the  curve $C$ is geometrically irreducible and its geometric genus is at  most $1$  by Theorem \ref{thmF} applied to the Albanese map of $C$.
Taking a complex uniformization of $C$ would provide a non-constant holomorphic map from $\C$ to $Y_{\Gamma}$, which is impossible  by  Lemma \ref{covers}(i)
which we can apply as  $\Gamma$ is torsion free. Therefore $Y_{\Gamma}$ is  Mordellic.
\end{proof}

\begin{proof}[Proof of Theorem \ref{picard-alternative}]

 The Lang locus of a quasi-projective irreducible variety $Z$ over a number field is defined as the Zariski closure of the union, over all number fields $k$,
of irreducible components of positive dimension of the Zariski closure of $Z(k)$. It is clear that  $Z$ is arithmetically Mordellic if and only if its Lang locus  is empty. The main  theorem in \cite{ullmo-yafaev} asserts that, for $\Gamma$ neat and sufficiently small,  the Lang locus of  $Y_\Gamma^\ast$  is either empty or everything.

By Corollary \ref{growth-cor} one can assume by further shrinking $\Gamma$ that $q(Y_\Gamma)>n$, and 
by \eqref{toroidal} we also have $q(X_\Gamma)>n$ for $X_\Gamma$ a  smooth toroidal compactification  of $Y_\Gamma$. By Corollary \ref{brodyan}  the Lang locus of $X_\Gamma$ is not everything, which forces the Lang locus of $Y_\Gamma^\ast$ to be empty.
 \end{proof}

\begin{proof}[Proof of Theorem \ref{picard-open}]
Let us first  show that $X_\Gamma$ is of general type, hence its canonical divisor $\mathcal{K}_X$ is big
in the sense of  \cite[Definition 1.1]{nadel}. Note that just like irregularity,
  the Kodaira dimension  cannot decrease when going to a finite covering.
By Holzapfel \cite[Theorem 5.4.15]{holzapfel-book} and Feustel \cite{feustel}
 the surface $X_{\Gamma_1(\fD)}$ is of general type for all $$D\notin \{3,4,7,8,11,15,19,20,23,24,31,39,47,71\}.$$
 Also by \cite[Proposition 4.13]{holzapfel},  $X_{\Gamma(N)}$ is of general type for all integers $N>2$, with the  possible exceptions  of $N=3$ and $N=4$ when $D=4$, implying in particular that  $X_{\Gamma(\fD^2)}$ is of general type
 for $D\in  \{3,4,7,8,11,19\}$.
Finally, the argument from {\it loc. cit.}  transports in a straight forward way to the case when the level is an ideal
 of $\fO$, yielding that the remaining varieties $X_{\Gamma(\fD)}$,  $D\in  \{15,20,23,24,31,39,47,71\}$, are of general type  as well.
 See  \cite{dzambic} where this has been carried out (it should be noted that his argument works also  when  $D=24$).

If $g=\sum_{i,j=1}^2 g_{i{\bar j}} dz_i{d\bar z_j}$ denotes the Bergman metric of $\h^2_\C$ viewed as the unit ball $\{z=(z_1, z_2)\in \C^2 \, , \, |z|<1\}$, normalized by requiring that $$\mathrm{Ric}(g)=\sum_{i,j=1}^2 -\frac{\partial^2 \log( g_{1\bar 1}g_{2\bar 2}-g_{2\bar 1}g_{1\bar 2})}{\partial {z_i} \partial {\bar z_j}} dz_i{d\bar z_j}=-g,$$ then the holomorphic sectional curvature is constant and equals $-4/3$ (see \cite[\S3.3]{GKK}), where  $g_{i\bar j}=\frac{3\left((1-|z|^2)\delta_{ij}+\bar z_iz_j\right)}{(1-|z|^2)^2}$.

By \eqref{toroidal} and  Proposition \ref{low} we have that $\Gamma$ is neat and $q(X_\Gamma)=q(Y_\Gamma)>2$. Corollary \ref{brodyan} then implies  that  $X_\Gamma(k)$ is not Zariski dense in $X_\Gamma$
for any finitely generated field extension $k$ of $\Q$ over which $X_{\Gamma}$ is defined. 
If $X_\Gamma(k)$ is infinite, arguing as in the proof of Theorem \ref{picard-compact}
shows that $X_\Gamma$ contains a geometrically irreducible curve $C$ whose geometric genus is at most one.  Now applying a result of Nadel  \cite[Theorem 2.1]{nadel} with $\gamma=1$ (so that $-\gamma \geq -4/3$), we see that the bigness of $\mathcal{K}_X$ implies that   $C$ is contained in the compactifying  divisor, which is a finite union of elliptic curves indexed by the cusps.
It follows that  $Y_\Gamma^\ast(k)$ is finite and that $X_\Gamma$  does not contain any rational curves at all, let alone just those of self intersection $-1$, hence it is a minimal surface of general type.
\end{proof}

\small

\bibliographystyle{siam}

\normalsize
\end{document}